\documentclass{amsart}

\usepackage{amsmath}
\usepackage{amssymb}
\usepackage{amsfonts}
\usepackage{enumerate}
\usepackage{vmargin}

\newcommand{\tens}{\otimes}

\newcommand{\bC}{{\mathbb{C}}}

\newcommand{\bN}{{\mathbb{N}}}
\newcommand{\bQ}{{\mathbb{Q}}}
\newcommand{\bR}{{\mathbb{R}}}

\newcommand{\bZ}{{\mathbb{Z}}}

  \newcommand{\C}{{\mathcal{C}}}

  \newcommand{\M}{{\mathcal{M}}}

  \newcommand{\R}{{\mathcal{R}}}

\renewcommand{\phi}{\varphi}
\newcommand{\upchi}{{\raise.35ex\hbox{\ensuremath{\chi}}}}
\newcommand{\eps}{\varepsilon}

\renewcommand{\leq}{\leqslant}
\renewcommand{\geq}{\geqslant}
\renewcommand{\le}{\leqslant}
\renewcommand{\ge}{\geqslant}

\newtheorem{thm}{Theorem}[section]
\newtheorem{defi}[thm]{Definition}
\newtheorem{prop}[thm]{Proposition}
\newtheorem{cor}[thm]{Corollary}
\newtheorem{lemma}[thm]{Lemma}
\newtheorem{rk}[thm]{Remark}

\begin{document}

\title{Fractional powers on noncommutative $L_p$ for $p<1$}
\date{}
\author[\'E. Ricard]{\'Eric Ricard}
\address{Normandie Univ, UNICAEN, CNRS, Laboratoire de Math{\'e}matiques Nicolas Oresme, 14000 Caen, France}
\email{eric.ricard@unicaen.fr}


\begin{abstract}
We prove that the homogeneous functional calculus associated to
$x\mapsto |x|^\theta$ or $x\mapsto {\rm sgn}\, (x) |x|^{\theta}$ for
$0<\theta<1$ is $\theta$-H\"older on selfadjoint elements of
noncommutative $L_p$-spaces for $0<p\leq\infty$ with values in
$L_{p/\theta}$. This extends an inequality of Birman, Koplienko and
Solomjak also obtained by Ando.
\end{abstract}
\maketitle
\section{Introduction}

 This note deals with the perturbation theory of functional calculus of
 selfadjoint operators on Hilbert spaces.  More precisely, given a
 function $f:\bR\to \bR$, the problem is to get a control of
 $\|f(x)-f(y)\|_{\mathfrak S}$ for some symmetric norm on selfadjoint
 operators in terms of possibly another norm $\|x-y\|_{{\mathfrak
   S}'}$.  This topic was developed from the 50's by the Russian
 school. Birman and Solomjak had a strong impact on it by the
 introduction of operator integrals in the 60's. Since then, this
 subject has been very active. Many mathematicians tried to enlarge
 the classes of functions $f$ or norms involved. The list would be too
 long, but we can quote Arazy \cite{A,AF}, Ando \cite{Ando}, and more
 recently the breakthroughs by Alexandrov-Peller \cite{AP2,AP},... and
 Potapov-Sukochev \cite{PS, PS2} and their coauthors \cite{PST},...
 Usually the results are stated for symmetric (quasi-)norms on compact
 operators, for instance the Schatten $p$-classes $S^p$ for $0<p\leq
 \infty$. Nevertheless the noncommutative integration theory in von
 Neumann algebras also gives a natural framework to study these
 questions.

 Our starting point is an inequality in \cite{BKS}, for any fully symmetric
 norm $\|.\|_{\mathfrak S}$ and any $0<\theta<1$, and $x,y$ positive
 operators on some Hilbert space, i.e $x,y\in B(H)^+$
$$\big\|x^\theta -y^\theta\big\|_{\mathfrak S}\leq \big\| \,
 |x-y|^\theta\,\big\|_{\mathfrak S}.$$

It was extended by Ando \cite{Ando} to any operator monotone function
$f:\bR^+\to \bR^+$ instead of $x\mapsto x^\theta$. Dodds and Dodds
\cite{DD} adapted the proof to semi-finite von Neumann algebras for all fully
symmetric norms.

In the case of Schatten classes, Birman Koplienko and Solomjak's result or Ando's proof actually give that for $p\geq \theta$ and $x,y\in B(H)^+$
\begin{equation*}\label{and}\big\|x^\theta -y^\theta\big\|_{p/\theta}\leq \big\| \,  x-y\,\big\|_{p}^\theta.\end{equation*}
This also holds for semi-finite Neumann algebras by \cite{DD} or \cite{PS2}. For
general von Neumann algebras, Kosaki got the case $p=\theta$ in \cite{K2}
with an extra factor, a full argument can be found in \cite{CPPR} or
\cite{R}.

Another remarkable extension was obtained in \cite{AP} for Schatten
$p$-classes when $1<p<\infty$; it is shown that for any
$\theta$-H\"older function $f$ on $\bR$, with $0<\theta<1$ and any
selfadjoint $x,y\in B(H)^{sa}$, one has 
\begin{equation}\label{and2}\| f(x)-f(x)\|_{p/\theta}\leq
C_{p,f} \|x-y\|_p^\theta.\end{equation}
 In particular this holds if
$f(x)=|x|^\theta$ or $f(x)={\rm sgn}(x)|x|^\theta$. For them, the arguments
can be adapted to general von Neumann algebras \cite{R} and one can also 
reach $p=1$ in \eqref{and2}.

 Surprisingly, when $p<1$ even for Schatten classes, very little is
 known. One can find some asymptotic estimates in \cite{BKS} or
 \cite{Rot2} but \eqref{and2} seems to be unknown. Weaker related
 inequalities were also recently obtained in \cite{Sob}.

Among  general results
Raynaud \cite{Ray} proved that $x\mapsto f(x)$ from $L_p$ to $L_{p/\theta}$ 
is uniformly continuous on balls for $f$ as above.
In \cite{PR}, for type II von Neumann algebras a strange quantitative estimate
was obtained for its modulus of continuity. 

Our main result is that \eqref{and2} holds for all $0<p\leq \infty$ for
both $f(x)=|x|^\theta$ and $f(x)={\rm sgn}(x)|x|^\theta$ and all von
Neumann algebras. We
hope that the techniques developed here may be useful for related
topics.

We do not address similar questions when $\theta>1$. When $p\geq 1$,
this is done in \cite{R} and when $p<1$, some local results can be
found in \cite{A,PST} for Schatten classes.

As usual to deal with such questions, one has to find norm estimates
for some Schur multipliers, this is done in the second section. Next,
they are used to derive the main result for semi-finite von Neumann
algebras. The argument heavily rests on homogeneity of $f$. To
generalize to type III algebras, one usually relies on the Haagerup
reduction principle, but it involves approximations using conditional
expectations that are not bounded when $p<1$ and it seems difficult to
use it in our situation.  The only available tool we have is to use
weak-type inequalities in semi-finite algebras to go to type III. The
situation is so particular here that this can be done quite easily in
section 4. We end up with some general remarks and extensions.

In the whole paper, we freely use noncommutative $L_p$-spaces. One may
use \cite{PX, P, terp} and \cite{FK} as general references. When $\tau$ 
is a normal faithful
trace on a von Neumann algebra $\M$, we use the classical definition of
noncommutative $L_p$ associated to $\M$
$$L_p(\M,\tau)=\{ x \in L_0(\M,\tau) \;|\;
\|x\|_p^p=\tau(|x|^p)<\infty\},$$ where $L_0(\M,\tau)$ is the space
$\tau$-measurable operators (see \cite{terp}).  When dealing with more
general von Neumann algebras, we rely on Haagerup's
construction. Given a normal faithful semi-finite weight $\phi_0$ on a
von Neumann algebra $\M$, Haagerup defined the noncommutative
$L_p$-space $L_p(\M,\phi_0)$ for $0< p\leq \infty$ (see
\cite{terp}). His definition is independent of $\phi_0$ (Corollary 38
in \cite{terp}). When $\phi_0$ is a normal faithful trace, his
definition is equivalent to the previous one (up to a complete
isometry) but the identifications are not obvious. Nevertheless for
most of our statements, we won't need the reference to $\phi_0$ or
$\tau$ so we may simply write $L_p(\M)$. When $0<p<1$, $L_p(\M)$ is a
$p$-normed space so that for all families $(a_k)$ in $L_p(\M)$, $\|\sum_{k=1}^n
a_k\|_p^p\leq  \sum_{k=1}^n\|a_k\|_p^p$.

We will use the notation
$S^p_{I,J}$ for the Schatten $p$-class on $B(\ell_2(J),\ell_2(I))$,
this is naturally a subspace of $L_p(B(\ell_2(I\cup J)),{\rm tr})$,
where ${\rm tr}$ is the usual trace. Thus by $S^p_{I,J}[L_p(\M)]$, we
will mean the corresponding subspace of $L_p(B(\ell_2(I\cup J))\otimes
\M,{\rm tr}\otimes \phi_0)$, one can think of it as matrices indexed
by $I\times J$ with coefficients in $L_p(\M)$. We will often use non
countable sets like $I=]0,1]$.

As usual, we denote constants in inequalities by $C_{p_i}$ if they
depend only on parameters $(p_i)$. They may differ from line to line.

\section{Schur multipliers}

A Schur multiplier with symbol $M=(m_{i,j})_{i\in I, j\in J}$ over
$M_{I,J}$, the set of matrices indexed by sets $I$ and $J$, is formally given by
$$S_M((a_{i,j})_{i\in I, j\in J})=M\circ A=(m_{i,j}a_{i,j})_{i\in I, j\in J}.$$
\begin{defi}
Given a matrix $M=(m_{i,j})_{i\in I, j\in J}$ of complex numbers, we
say that $M$ defines a $p$-completely bounded Schur multiplier for
some $0<p\leq \infty$ if the map $S_M\tens Id_{L_p(\M)}$ on
$S_{I,J}^p[L_p(\M)]$ is bounded for all von Neumann algebra $\M$ and
we put $\|M\|_{pcb}= \sup_{\M} \|S_M\tens Id_{L_p(\M)}\|$.
\end{defi}

\begin{rk}{\rm
For $1<p\neq 2<\infty$, this is not exactly the usual definition of
complete boundedness but it is formally stronger. Indeed an
unpublished result of Junge states that $\|S_M\tens
Id_{L_p(\M)}\|\leq \|S_M\|_{cb}$ if $\M$ is a QWEP von Neumann algebra.}
\end{rk}

We start by easy examples that can be found in \cite{AP2}.

\begin{lemma}\label{schur1}
Let $(\alpha_k)\in \ell_p(\bZ)$ for $0<p\leq1$ and assume that $(f_k)\in \ell_\infty(I)^\bZ$ and
$(g_k)\in \ell_\infty(J)^\bZ$ are bounded families. Then $M$ given by
$m_{i,j}=\sum_k \alpha_k f_k(i)g_k(j)$ is a $p$-completely bounded Schur multiplier with $$\|M\|_{pcb}\leq \|(\alpha_k)\|_p \sup_k \|f_k\|_\infty.\|g_k\|_\infty.$$
\end{lemma}

\begin{proof}
It is clear that a rank one symbol $M_k=(f_k(i)g_k(j))_{i\in I,j\in J}$ defines a
$p$-completely bounded Schur multiplier with norm $\|f_k\|_\infty.\|g_k\|_\infty$
for all $p$ and $k$. The result
then follows by the $p$-triangular inequality.
\end{proof}

We will often use permanence properties of $pcb$-Schur multipliers.

\begin{lemma}\label{perm}
Assume $M=(m_{i,j})_{i\in I, j\in J}$ is a $p$-completely bounded Schur multiplier, 
then 
\begin{itemize}
\item $M'=(m_{i,j})_{i\in I', j\in J'}$ with $I'\subset I, \, J'\subset J$, then 
$\|M'\|_{pcb}\leq \|M\|_{pcb}$.
\item $M'= (m_{i,j})_{(i,k)\in I\times K, (j,l)\in J\times L}$ for any non empty sets $K, L$, then $\|M'\|_{pcb}=\|M\|_{pcb}$.
\end{itemize}
\end{lemma}
\begin{proof}
We view $\ell_2(I')$, $\ell_2(J')$ as subspaces of $\ell_2(I)$,
$\ell_2(J)$. Let $P=(1_{I'}(i)1_{J'}(j))_{i\in I, j\in J}$, this is a
rank one $p$-completely bounded Schur multiplier with norm 1, and $S_{M'}$ 
coincides with the restriction of $S_P\circ S_{M}$ to matrices indexed by $I'\times J'$.

The second point is classical using tensorisation with $B(\ell_2(L),\ell_2(K))\otimes \M$ instead of $\M$.
\end{proof}

Since finitely supported matrices are dense in $S_{I,J}^p[L_p(\M)]$, we also have
\begin{lemma}\label{approx}
If $(M_n)\in M_{I,J}^\bN$ is  a bounded sequence of $pcb$-Schur multipliers converging
pointwise to some $M$, then $M$ is also a  $pcb$-Schur multiplier with 
$\|M\|_{pcb}\leq \lim \|M_n\|_{pcb}$.
\end{lemma}

\begin{rk}{\rm 
  A $p$-completely bounded Schur multiplier $M$ for $p\leq 1$ is
  automatically $q$-completely bounded for $p<q\leq \infty$. Indeed,
  the extreme points of the unit ball of $S^1$ are rank one
  matrices, but for those matrices the $S^1$ and $S^p$ norms coincide. Thus
 $S_M$ must be bounded on $S^1$. But bounded Schur
  multipliers on $S^1$ are automatically $1$-completely bounded (see \cite{P}).
 Thus we get the result on $S^q$
  for all $1\leq q\leq \infty$ by complex interpolation and duality. The case $
p<q\leq 1$ also follows by interpolation.}
\end{rk}

 The following is a suitable adaptation of classical arguments (see
 \cite{BKS, Rot, Rot2}). We use the measured space 
$L_2( [0,2\pi]^2, \frac 1{(2\pi)^2}dm_2)$ where $m_2$ is the Lebesgue measure.

\begin{lemma}\label{sob}
Let $K:[0,2\pi]\times [0,2\pi]\to \bC$ be a $2\pi$-periodic continuous
function  such that for any $d\geq 0$ $\frac
{\partial^{d+1}}{(\partial y )^d\partial x}K$ is 
continuous.  Then $M=(K(x,y))_{x,y\in [0,2\pi]}$ is a $p$-completely
bounded Schur multiplier for all $0<p\leq 1$ with for  $d>1/p$
$$\|M\|_{pcb}\leq C\Big(\frac 2 {dp-1}+2\Big)^{1/p} \Big(\big\|\frac
      {\partial^{d+1}}{(\partial y )^d\partial
        x}K\big\|_2+\big\|\frac {\partial^{d}}{(\partial y
        )^d}K\big\|_2 + \big\|\frac
      {\partial}{\partial x}K\big\|_2+\big\|K\big\|_2\Big),$$
where $C$ is a universal constant. Moreover if $M_i=(K_i(x,y))_{x,y\in [0,2\pi]}$ is a family of matrices  as above, indexed by $i\in I$, 
then $$M=(K_i(x,y))_{(x,i)\in [0,2\pi]\times I, y\in [0,2\pi]}$$ or its transpose satisfies
$$\|M\|_{pcb}\leq  C\Big(\frac 2 {dp-1}+2\Big)^{1/p} \sup_i\Big(\big\|\frac
      {\partial^{d+1}}{(\partial y )^d\partial
        x}K_i\big\|_2+\big\|\frac {\partial^{d}}{(\partial y
        )^d}K_i\big\|_2 + \big\|\frac
      {\partial}{\partial x}K_i\big\|_2+\big\|K_i\big\|_2\Big).$$
\end{lemma}
Of course, this is relevant only if the above $\sup$ is finite.
\begin{proof}
We rely on Fourier expansions, put $e_k(x)=e^{{\rm i}kx}$ and
$h_{k,l}(x,y)=e_k(x)e_l(y)$ . As $(h_{k,l})_{k,l\in \bZ}$ is an
orthonormal basis in $L_2( [0,2\pi]^2, \frac 1{(2\pi)^2}dm_2)$, we
have the equality in $L_2$, $K= \sum_{l,k\in \bZ} \alpha_{k,l}
h_{k,l}$ where $\alpha_{k,l}=\langle
K,h_{k,l}\rangle=\frac 1{(2\pi)^2}\int_{0}^{2\pi}\int_{0}^{2\pi} K(x,y) e^{-{\rm
    i}kx}e^{-{\rm i}ly}dydx$. 

Assume for the moment that $\l\neq0$. Integrating by part in $y$, we get for
$d\geq 0$, $\alpha_{k,l}= \frac {1} {({\rm i}l)^d}
\langle\frac {\partial^{d}}{(\partial y )^d}K,h_{k,l}\rangle$. Let
$\beta_{k,l}=({\rm i}l)^d\alpha_{kl}$. When $k\neq 0$, another integration by part  with
respect to $x$ gives $\beta_{k,l}=\frac 1 {{\rm i}k} \frac 1{(2\pi)^2}\int_{0}^{2\pi}\int_{0}^{2\pi}
\frac {\partial^{d+1}}{(\partial y )^d\partial x}K(x,y) e^{-{\rm i}kx}e^{-{\rm
    i}ly}dydx$. The Cauchy-Schwarz inequality gives that
 $$\sum_{k\neq 0} |\beta_{k,l}|\leq \Big(\sum_{k\neq 0} \frac 1
{k^2}\Big)^{1/2}\Big(\sum_{k\neq 0} \big|\langle \frac
{\partial^{d+1}}{(\partial y )^d\partial x}K,
h_{k,l}\rangle\big|^2\Big)^{1/2}\leq C \Big\|\frac
{\partial^{d+1}}{(\partial y )^d\partial x}K\Big\|_2.$$ Thus
$\sum_{k\in \bZ} |\beta_{k,l}|\leq C \big\|\frac
{\partial^{d+1}}{(\partial y )^d\partial x}K\big\|_2+\big\|\frac
{\partial^{d}}{(\partial y )^d}K\big\|_2=C_d$ independent from $l$ and
one can define a continuous function $f_l(x)=\sum_{k\in \bZ} \frac
1{{\rm i}^d}\beta_{k,l} e_k(x)$ bounded by $C_d$. 

In the same way to deal with $l=0$,
$f_0(x)=\sum_{k\in \bZ} \alpha_{k,0} e_k(x)$ is a continuous function. Indeed as above
$$\sum_{k\in \bZ} |\alpha_{k,0}|\leq \Big(\sum_{k\neq 0} \frac 1
{k^2}\Big)^{1/2}\Big(\sum_{k\neq 0} \big|\langle \frac
{\partial}{\partial x}K,
h_{k,0}\rangle\big|^2\Big)^{1/2}   + |\alpha_{0,0}|,$$
$f_0$ is bounded  by $C_0=C \big\|\frac
{\partial}{\partial x}K\big\|_2+\big\|K\big\|_2$. 

 Choosing $d> \frac 1 p$, we can conclude to the
pointwise equality
\begin{equation}\label{fact} K(x,y)=f_0(x)e_0(y)+ \sum_{l\neq 0} \frac 1 {l^d}  f_l(x) e_l(y).\end{equation}
The result follows directly from Lemma \ref{schur1} by choosing
$I=\bZ$, $\alpha=(1_{k\neq 0}\frac 1 {k^d}+ 1_{k=0})_k\in \ell_p(\bZ)$, $f_k$ as
above and $g_k=e_k$.  Obviously $\sup_k\|g_k\|_\infty=1$, $\sup_k\|f_k\|_\infty\leq C_d+C_0$ and $\|\alpha\|_p\leq \Big(\frac
2 {dp-1}+2\Big)^{1/p}$.

The second statement also follows from Lemma \ref{schur1} since
 in \eqref{fact}, the
factorization in $y$ and the sequence $\alpha$ is independent from $K_i$.
We have  $K_i(x,y)=f_0^i(x)e_0(y)+ \sum_{l\neq 0} \frac 1 {l^d}  f_l^i(x) e_l(y)$,
hence we can again take $\alpha=(1_{k\neq 0}\frac 1 {k^d}+1_{k=0})_k\in \ell_p(\bZ)$, 
$f_k(x,i)=f_k^i(x)$  and $g_k(y)=e_k(y)$. We also have $\sup_k\|f_k\|_\infty\leq C_d+C_0$.

The same holds for the transpose of $M$ as the condition in Lemma \ref{schur1}
is invariant by transposition.
\end{proof}
The Sobolev constant (of order $d$) for $K$ will mean the quantity $$\big\|\frac
      {\partial^{d+1}}{(\partial y )^d\partial
        x}K\big\|_2+\big\|\frac {\partial^{d}}{(\partial y
        )^d}K\big\|_2 + \big\|\frac
      {\partial}{\partial x}K\big\|_2+\big\|K\big\|_2.$$
Let $\theta\in ]0,1[$, for $x,y\geq 0$ recall that 
\begin{equation}\label{intfor}\frac {x^\theta- y^\theta}{x-y}= \int_0^1  \frac\theta{(tx+(1-t)y)^{1-\theta}} dt,\end{equation}
where the left hand side has to be understood as $\theta x^{\theta-1}$ if $x=y$.

\begin{cor}\label{theta2}
The matrix $N=\Big(\frac {x^\theta- y^\theta}{x-y}\Big)_{x\geq0 ,y\in [1,2]}$ defines a
 $p$-completely bounded Schur multiplier for $0<p\leq 1$ with 
$\|N\|_{pcb}\leq C_{p}$ for some constant depending only on $p$.
\end{cor}

\begin{proof}
First we start by showing that
  $\Big(\frac {x^\theta-
  y^\theta}{x-y}\Big)_{0\leq x\leq 1/2 ,y\in [1,2]}$ is a $pcb$-Schur multiplier.

We fix a $\C^\infty$ function $\phi:[-\pi,\pi]\to [0,1]$
         with support in $[-1/4, 3/4]$ that is identically
        1 on $[0,1/2]$ and another $\C^\infty$ function $\psi:[0,2\pi]\to [0,1]$
        with support in $[7/8, 3]$ that is identically 1 on $[1,2]$.
        We define $K(x,y)=\phi(x)\psi(y)\frac {1}{x-y}$
        on $[-\pi,\pi]\times[0,2\pi]$. It is $\C^\infty$ and can be extended
        to a $2\pi$-periodic $\C^\infty$ function. Thus Lemma \ref{sob} and a
        restriction  yield that $\Big(\frac {1}{x-y}\Big)_{0\leq x\leq
          1/2 ,y\in [1,2]}$ is a $pcb$-Schur multiplier. Then
        one just need to compose it with
        $\big({x^\theta-y^\theta}\big)_{0\leq x\leq 1/2 ,y\in [1,2]}$
        which is also clearly a $pcb$-Schur multiplier by Lemma \ref{schur1}.

        Next we show that $\Big(\frac {x^\theta- y^\theta}{x-y}\Big)_{x\geq 1/2 ,y\in
          [1,2]}$ is also  a $pcb$-Schur multiplier.
        
        This time we fix a $\C^\infty$ function $\phi:[0,2\pi]\to [0,1]$
with support in $[1/4, 3]$ that is identically 1 on $[1/2,2]$.

For $i\geq 0$, one uses $K_i(x,y)=\phi(x)\phi(y)\frac
{(x+i)^\theta-y^\theta}{(x+i)-y}$ on $[0,2\pi]^2$. It is clear that
$K_i$ can be extended to a $\C^\infty$ $2\pi$-periodic function.  By
construction, for $x$ and $y$ in the support of $\phi$ and $t\in[0,1]$, 
the smallest value of $t(x+i)+(1-t)y$ is bigger 
than $1/4$. Thus, the formula \eqref{intfor} shows that any derivative of
order $l$ of $\frac {x^\theta- y^\theta}{x-y}$ on the support of $K_i$
is bounded by $4^{l+1-\theta}\theta(1-\theta)...(l-\theta)$. Thus
using the chain rule, one sees that $K_i$ and its derivatives up to
order $d+1$ are uniformly bounded independently of $i$ and
$\theta$. Thus the same holds for the Sobolev constant in Lemma
\ref{sob} for $K_i$.

Lemma \ref{sob} gives that
$\big(K_i(x,y)\big)_{(x,i)\in [0,2\pi]\times \bN,y\in
  [0,2\pi]}$ is $pcb$. By Lemma \ref{perm}, we can conclude by restricting $x$ to $[1/2,3/2[\times \bN\simeq [1/2,\infty[$ (via $(x,i)\mapsto
        x+i$) and $y$ to $[1,2]$.

        The Corollary follows by gluing the two pieces together.
\end{proof}

\begin{cor}\label{theta4}
For $k\in \bZ$, the matrix $M_{k}=\Big(\frac {x^\theta-
  y^\theta}{x-y}\Big)_{x\geq 0,y\in [2^{-k-1},2^{-k}]}$ is a
$p$-completely bounded Schur multiplier for $0<p\leq 1$ with
$\|M_k\|_{pcb}\leq C_{p}2^{-k(\theta-1)}$ for some constant depending only on $p$.
\end{cor}

\begin{proof}
This is obvious by homogeneity from Corollary \ref{theta2} with a
change of variables  $x\leftrightarrow 2^{-k-1}x$, $y\leftrightarrow 2^{-k-1}y$.
\end{proof}

\begin{rk}\label{exchange}{\rm
One can exchange the roles of $x$ and $y$.}
\end{rk}

It will be convenient to redefine $M_{-1}$, gathering all $k\leq -1$:
\begin{cor}\label{theta5}
The matrix 
$M_{-1}=\Big(\frac {x^\theta- y^\theta}{x-y}\Big)_{x\geq 0,y\geq 1}$ is a
 $p$-completely bounded Schur multiplier for $0<p\leq 1$ with 
$\|M_{-1}\|_{pcb}\leq C_{p}\big(\frac 1 {1-\theta}\big)^{1/p}$ for some constant depending only on $p$.
\end{cor}

\begin{proof}

Writing $[1,\infty[=\cup_{k\geq 0} [2^k,2^{k+1}[$ and using the previous Corollary
for each piece, this follows from 
the $p$-triangular inequality as $(2^{k(\theta-1)})_{k\geq 0}\in \ell_p(\bN)$.
Since $\|(2^{k(\theta-1)})_{k\geq 0}\|_p\leq c_p (1-\theta) ^{-1/p}$ for some constant $c_p$. We get that $\|M_{-1}\|_{pcb}\leq   c_p C_p(1-\theta) ^{-1/p}$ where $C_p$ comes from \ref{theta4}.
\end{proof}

\begin{rk}{\rm
   The kernel in formula \eqref{intfor} is positive definite 
because for  $x,y>0$  
$\frac {x^\theta- y^\theta}{x-y}= c_\theta\int_{\bR_+} t^{\theta}\frac 1{x+t}\frac 1{y+t}{dt}$ for some $c_\theta>0$. Using this fact and similar arguments, one can check that
  there is some $C$ so that $\|M_{-1}\|_{pcb}\leq C$ for all $0<\theta<1$ and all $p\geq 1$. }
  \end{rk}

We now turn to another family of multipliers.

\begin{cor}\label{theta+}
For $a\geq 1$, the matrix 
$H_a=\Big(\frac {1}{a+x+y}\Big)_{x,y\in [0,1]}$ is a $p$-completely bounded Schur multiplier for $0<p\leq 1$ with 
 $\|H_a\|_{pcb}\leq C_{p}/a$ for some constant $C_p$ depending only on $p$. 
\end{cor}

\begin{proof}
As for Corollary \ref{theta2} take a smooth function
$\phi:[-\pi,\pi]\to [0,1]$ that is supported on $[-1/4,5/4]$ such that
$\phi(t)=1$ for $t\in [0,1]$. Define $K(x,y)=\frac 1{a+x+y}
\phi(x)\phi(y)$ on $[-\pi,\pi]$ and make it $2\pi$-periodic so that it is
$\C^\infty$. Using the chain rule, one easily sees that the Sobolev norms from
Lemma \ref{sob} are dominated by $C_p/a$.
\end{proof}

\begin{cor}\label{theta+2}
Given $a,b\geq 0$ with $a+b>0$, one has 
$$\Big\|\Big(\frac {x^\theta\pm y^\theta}{x+y}\Big)_{x\geq a,y\geq b}\Big\|_{pcb}\leq C_{p}\big(\frac 1 {1-\theta}\big)^{1/p}
\max \{a,b\}^{\theta-1},$$ for some constant $C_{p}$ depending on $p$.
\end{cor}
\begin{proof}
Without loss of  generality we may assume $a\geq b$. 

By a change of variable $x\leftrightarrow a(1+x)$ and $y\leftrightarrow a(t+y)$, with $t=b/a$, it boils down to show that
$\Big\|\Big(\frac {(1+x)^\theta\pm (t+y)^\theta}{1+t+x+y}\Big)_{x\geq 0,y\geq
  0}\Big\|_{pcb}$ is bounded independently of $t\in]0,1]$. 

 We use a dyadic decomposition related to $\max\{x,y\}$.  Assume $x\in
 I_k=[2^k,2^{k+1}[$ and $y\in J_k=[0,2^{k}[$.  
Then by homogeneity and a change of variables $x\leftrightarrow
 2^k(x+1)$, $y\leftrightarrow 2^ky$
$$\Big\|\Big(\frac {1}{1+t+x+y}\Big)_{x\in I_k,y\in J_k}\Big\|_{pcb}=
 2^{-k}\Big\|\Big(\frac {1}{1+2^{-k}(1+t)+x+y}\Big)_{x\in [0,1[,y\in
       [0,1[}\Big\|_{pcb}.$$ Setting $a=1+2^{-k}(1+t)\in [1,3]$ and
         using Corollary \ref{theta+} the latter multiplier is bounded
         by $C_p 2^{-k}$.  The multiplier $\Big(((1+x)^\theta\pm
         (t+y)^\theta)1_{x\in I_k,y\in J_k}\Big)$ is bounded by a
         fixed multiple of $2^{k\theta}$.  Thus $\Big\|\Big(\frac
         {(1+x)^\theta\pm (t+y)^\theta}{1+t+x+y}\Big)_{x\in I_k,y\in
           J_k }\Big\|_{pcb}\leq C_p 2^{-k(1-\theta)}$. A similar
         estimate holds for the same symbol if $x\in J_{k+1}$ and
         $y\in I_k$ or $x,y\in [0,1[$ (with $k=0$). The sets
 $[0,1[^2$, $J_{k+1}\times I_k$, $I_k\times J_k$ for $k\geq 0$ form a partition
of $[0,\infty[^2$ into product sets. Thus, 
the $p$-triangular inequality gives 
$$\Big\|\Big(\frac {(1+x)^\theta\pm (t+y)^\theta}{1+t+x+y}\Big)_{x\geq 0,y\geq
  0}\Big\|_{pcb}^p\leq C_p^p + 2\sum_{k\geq 0}C_p^p  2^{-p(1-\theta)k}\leq C_p^p \frac 1 {1-\theta}.$$
\end{proof}

\section{Ando's inequality in semi-finite algebras}

Schur multipliers are intimately related to perturbations of the
functional calculus of selfadjoint operators. Illustrations can be found in
\cite{BS, A, AP, AP2, PS, PS2, PST} and many other references. 

Indeed let $f:\bR \to \bR$ be a (continuous) function.  Assume that
$x,y$ are selfadjoint elements in some semi-finite von Neumann algebra
$(\M,\tau)$  with finite discrete spectra $(x_i)_{i\in I}$ and $(y_j)_{j\in
  J}$ and associated spectral projections $p_i\in \M$ and $q_j\in\M$. As
$x=\sum_i x_i p_i$ and $y=\sum_j y_i q_j$, with $\sum_i p_i=\sum_j q_j=1$, one has
\begin{equation}\label{intdif} 
f(x)-f(y)=\sum_{i,j} p_i (f(x_i)-f(y_j))q_j=\sum_{i,j} p_i 
\frac{f(x_i)-f(y_j)}{x_i-y_j} (x-y)q_j,\end{equation}
where $\frac{f(u)-f(v)}{u-v}$ can take any value if $u=v\in \bR$.

 The map $T:z\mapsto \sum_{i,j} p_i
 \frac{f(x_i)-f(y_j)}{x_i-y_j} zq_j$ is very close to be a Schur multiplier
with symbol the divided differences of $f$.

\begin{lemma}\label{sf}
For any $0<p\leq \infty$, let $M=(m_{i,j})$ be a $p$-completely bounded Schur multiplier then the following map  $T_M:\M\to \M$
$$T_{M}(z)=\sum_{i,j} m_{i,j} p_i  zq_j$$
extends to a bounded map on $L_p(\M)$ with $\|T_M\|_{L_p(\M)\to L_p(\M)}\leq \|M\|_{pcb}$.
\end{lemma}

\begin{proof}
Recall that we assume that $I$ and $J$ are finite, $(p_i)_{i\in I}$  and $(q_j)_{j\in J}$ are orthogonal projections in $\M$ summing up to 1. The maps $\pi : \M\to B(\ell_2(J),\ell_2(I))\tens \M$ and $\rho: B(\ell_2(J),\ell_2(I))\tens \M\to \M$ defined by 
$$\pi(z)= (p_izq_j)_{i\in I, j\in J}=
 \left(\begin{array}{c}p_1\\ \vdots
   \\ p_{|I|}\end{array}\right) z 
\left(\begin{array}{ccc} q_1 & \dots& q_{|J|}\end{array}\right)$$
$$\rho((X_{i,j}))=\sum_{i,j} p_iX_{i,j}q_j=\left(\begin{array}{ccc}
  p_1 & \dots& p_{|I|}\end{array}\right)(X_{i,j})
\left(\begin{array}{c}q_1\\ \vdots \\ q_{|J|}\end{array}\right) $$
clearly extend to contractions at $L_p$-levels respectively on $\M$ and $B(\ell_2(J),\ell_2(I))\tens \M$ since the column and
row matrices in the above products are contractions. Since
$T_M= \rho \circ (S_M\otimes Id_{L_p(\M)}) \circ \pi$, we get the lemma.  
\end{proof}

 We want to deal with  homogeneous functions of
selfadjoint operators, namely $f(x)=|x|^\theta$ or 
$f(x)={\rm sgn} (x) |x|^\theta=x_+^\theta-x_-^\theta$.

We aim to prove the following theorem that we call Ando's inequality.

\begin{thm}\label{ando2}
Let $0<\theta<1$ and $0<p\leq \infty$ then there exists $C_{p,\theta}$ so that for any  von Neumann algebra $\M$, and $x,y\in L_p(\M)^{sa}$ one has
\begin{equation}\label{main}\| |x|^\theta-|y|^\theta\|_{p/\theta}\leq C_{p,\theta}
\|x-y\|^\theta_p, \qquad \|{\rm sgn} (x) |x|^\theta-{\rm sgn} (y) |y|^\theta
\|_{p/\theta}\leq
C_{p,\theta} \|x-y\|^\theta_p.\end{equation}
\end{thm}

The reduction from type III to type II will be explained in the next
section so we only deal with semi-finite algebras here.

\begin{proof}

The result for $p=\infty$ follows from Theorem 4.1 in  \cite{AP3}.

The result for $f(x)={\rm sgn} (x) |x|^\theta$ when $1\leq p<\infty$
can be found in \cite{R}. The absolute value map $x\mapsto |x|$ is
bounded on $L_q(\M)$ provided that $1<q<\infty$. This is a result by
Davies \cite{Dav} for Schatten classes, that can be extended to all
semi-finite von Neumann algebras (see Remark 6.2 in \cite{CPSZ} for instance).
Thus for $1\leq p<\infty$, with $q=p/\theta$, we have an estimate $\|
|x|^\theta-|y|^\theta\|_{p/\theta}\leq C_{p/\theta} \|{\rm sgn} (x) |x|^\theta-{\rm
  sgn} (y) |y|^\theta \|_{p/\theta}$. Therefore the
 Theorem holds for $f(x)=|x|^\theta$ using it for 
$f(x)={\rm sgn} (x) |x|^\theta$ when $1\leq p<\infty$. 

\smallskip

We just need to
prove the Theorem for $p< 1$.

\smallskip

If the result holds for the couples $(p,\theta_1)$ and $(p/\theta_1,\theta_2)$, it also holds for $(p,\theta_1\theta_2)$. Indeed for instance if
$$\||x|^{\theta_1}- |y|^{\theta_1}\|_{p/\theta_1}\leq C_{p,\theta_1}\|x-y\|_p^{\theta_1} \quad\textrm{and} \quad
\||z|^{\theta_2}- |t|^{\theta_2}\|_{p/(\theta_1\theta_2)}\leq C_{p/\theta_1,\theta_2}\|z-t\|_{p/\theta_1}^{\theta_2},$$
then one gets with $z=|x|^{\theta_1}$, $t=|y|^{\theta_1}$:
$$\||x|^{\theta_1\theta_2}-
|y|^{\theta_1\theta_2}\|_{p/(\theta_1\theta_2)}\leq
C_{p/\theta_1,\theta_2}C_{p,\theta_1}^{\theta_2}\|x-y\|_p^{\theta_1\theta_2}.$$
Hence $C_{p,\theta_1\theta_2}\leq
C_{p/\theta_1,\theta_2}C_{p,\theta_1}^{\theta_2}$. By this
transitivity, we reduce the proof to $p\leq\theta$; indeed if $p>\theta$, then 
$(p,\theta)=(p.1, p.\theta/p)$ and the estimate follows from that for 
$(p,p)$ and $(1,\theta/p)$.

\smallskip

 We do it only for $f(t)=|t|^\theta$ as the other case is similar.

\smallskip

We prove the inequality from case to case regarding $\M$ and the values of $x$ and $y$.

\medskip

{\noindent \it Case 1:} We assume  that $x, y\in \M^{sa}$ with finite
discrete spectra and 
$\|x-y\|_\infty\leq 2$ and $\|x-y\|_{p/2}\leq 2$. We prove that $\| |x|^\theta-|y|^\theta\|_{p/\theta}\leq C_{p,\theta}$ for some $C_{p,\theta}>0$.

We denote by $x=\sum_{i\in I} x_i p_i$ and $y=\sum_{i\in J} y_j q_j$ the spectral decompositions of $x$ and $y$, so that $p_i=1_{\{x_i\}}(x)$ and $q_i=1_{\{y_i\}}(y)$.

\smallskip

 We will rely on the formula \eqref{intdif}, decompose
 $|x|^\theta-|y|^\theta=\sum_{i,j=-1}^1 a_i (|x|^\theta-|y|^\theta)b_i$
where $a_{-1}=1_{]\infty,0[}(x)$, $a_0=1_{\{0\}}(x)$ and $a_1=1_{]0,\infty[}(x)$ and similarly for $b_i$.

We use dyadic decompositions. Set $I_k=[2^{-k-1},2^{-k}[$,
    $J_k=]0,2^{-k-1}[$ for $k\geq 0$ and $I_{-1}=[1,\infty[$,
        $J_{-1}=]0,1[$, $J_{-2}=]0,\infty[$.  The definition is made
        so that the sets $I_k\times J_{k-1}$ and $J_k\times I_k$ are
        disjoint with union $\{(x,y)\in
      ]0,\infty[^2\;|\;\max\{x,y\}\in I_k\}$. Hence we have a
        partition $]0,+\infty[^2=\cup_{k\geq -1} (I_k\times
        J_{k-1}\cup J_k\times I_k)$. Define accordingly the maps
        $T^1_k(z)=\sum_{x_i\in I_k, y_j\in J_{k-1}} p_i
        \frac{x_i^\theta-y_j^\theta}{x_i-y_j} zq_j$,
        $T^2_k(z)=\sum_{x_i\in J_k, y_j\in I_k} p_i
        \frac{x_i^\theta-y_j^\theta}{x_i-y_j} zq_j$.  For any $0<a\leq
        1$, Lemma \ref{sf} says that $\|T_k^1\|_{L_a(\M)\to L_a(\M)}$ is
        dominated by
        $\Big\|\Big(\frac{x_i^\theta-y_j^\theta}{x_i-y_j}\Big)_{x_i\in
          I_k, y_j\in J_{k-1}}\Big\|_{acb}$. By Lemma \ref{perm}, this
        norm is smaller than
        $\Big\|\Big(\frac{x^\theta-y^\theta}{x-y}\Big)_{x\in \overline
          I_k,y\geq 0}\Big\|_{acb}$. Using Remark \ref{exchange} and
        Corollary \ref{theta4} for $k\geq 0$, this norm is bounded by
        $C_a 2^{-k(\theta-1)}$. When $k=-1$, using Corollary
        \ref{theta5} instead, we also get a bound by
        $C_{a,\theta}$. We have similar estimates for $\|T_k^2\|_{L_a(\M)\to L_a(\M)}$.

    Put $r_k=1_{I_k}(x)$, $s_k=1_{I_k}(y)$
    as well as $u_k=1_{J_{k}}(x)$, $v_k=1_{J_{k-1}}(y)$, we can write (note that the sums are actually finite) 
\begin{eqnarray*}a_1(|x|^\theta-|y|^\theta)b_1&=&\sum_{x_i>0, y_j>0} p_i
  \frac{x_i^\theta-y_j^\theta}{x_i-y_j} (x-y)q_j\\ & =&\sum_{k\geq -1}  T^1_k(x-y)+T_k^2(x-y)  \\&=&
  \sum_{k\geq -1}  T^1_k(r_k(x-y)v_k)+T_k^2(u_k(x-y)s_k).  
    \end{eqnarray*}
We use the above norm estimates for  $T^j_{k}$  with $a={p/\theta}$
 to get $\|T_k^j\|_{L_{p/\theta}(\M)\to L_{p/\theta}(\M)}\leq C_{p,\theta}
    2^{k(1-\theta)}$.  By the $p/\theta$-triangular inequality
$$\|a_1(|x|^\theta-|y|^\theta)b_1\|_{p/\theta}^{p/\theta}\leq C_{p,\theta}^{p/\theta}
\sum_{k\geq -1} 2^{k(1-\theta)p/\theta} (\|r_k (x-y)v_k\|_{p/\theta}^{p/\theta}+\|u_k 
(x-y)s_k\|_{p/\theta}^{p/\theta}) .$$
But by definition $0\leq r_k x, u_k x, yv_k, ys_k\leq 2^{-k}$ for $k\geq 0$
and $\|x-y\|_\infty\leq 2$, so that we have 
$\|r_k (x-y)v_k\|_\infty\leq \|r_k x-yv_k\|_\infty\leq 2^{-k}$, and also
$\|u_k x-ys_k\|_\infty\leq 2^{-k}$ for all $k\geq -1$ . But also
$\|r_k(x-y)v_k\|_{p/2},\|u_k(x-y)s_k\|_{p/2}
\leq \|x-y\|_{p/2}\leq 2$.  Thus as
$\theta/p=(\theta/2).2/p+(1-\theta/2)/\infty$, by the H\"older
inequality $\|r_k (x-y)v_k\|_{p/\theta}, \|u_k (x-y)s_k\|_{p/\theta}\leq 2.2^{-k(1-\theta/2)}$. This is enough to conclude that
$$\|a_1(|x|^\theta-|y|^\theta)b_1\|_{p/\theta}^{p/\theta}\leq 4C_{p,\theta}^{p/\theta} \sum_{k\geq -1} 2^{-kp/2}.$$

To deal with $ a_{1}(|x|^\theta-|y|^\theta)b_{-1}$, one can do exactly the same using Corollary \ref{theta+2}
as the Schur multipliers involved have shape $\Big( \frac
{a^\theta-b^{\theta}}{a+b}\Big)_{a\in I_k,b\in J_{k-1}}$, $\Big( \frac
{a^\theta-b^{\theta}}{a+b}\Big)_{a\in J_k,b\in I_k}$.

The terms $a_{-1}(|x|^\theta-|y|^\theta)b_1$ and $a_{-1}(|x|^\theta-|y|^\theta)b_{-1}$
can be treated similarly.

It is a well known fact that complex interpolation remains valid for
$L_p(\M,\tau)$ in the range $0<p\leq 1$, see Lemma
2.5 in \cite{PR} for what we need. Using it, one easily deals with the remaining terms
as for instance
$$\|a_0(|x|^\theta-|y|^{\theta})\|_{p/\theta}=\|a_0|y|^{\theta}\|_{p/\theta}\leq \|a_0|y|\|_p^\theta=\|a_0 y \|_p^\theta=\|a_0(x-y)\|_p^\theta\leq 2.$$

Gluing the pieces together, we find a constant $C_{p,\theta}$ so that \eqref{main}
holds in Case 1.
 
\medskip

 {\it\noindent  Case 2:} We assume  $\M$ finite, $x, y\in \M^{sa}$,   $y=x+q$ for some
projection $q$ with $\|q\|_p= 1$, that is $\tau(q)=1$. 

\smallskip

Consider $(x_n)$ a sequence in the von Neumann algebra generated by
$x$  that $|x_n|\leq |x|$,  $\|x_n-x\|_\infty\to 0$ and each $x_n$ has a finite spectrum, and
similarly for $(y_n)$ and $y$. 

Obviously $\||x_n|^\theta- |x|^\theta\|_{p/\theta}\to 0$ by the
dominated convergence theorem as $\M$ is finite (similarly for $y$). As
$\|x_n-y_n\|_t\to \|q\|_t=1$ for $t=\infty,p/2$, $x_n$ and $y_n$ satisfy the assumptions of Case 1 for $n$ big enough. Going to the limit in \eqref{main} for $x_n,y_n$, we get \eqref{main} for $x,y$ with the same $C_{p,\theta}$.

\medskip

{\noindent \it Case 3:}  We assume   $\M$ finite, $x, y\in \M^{sa}$  and $y=x+tq$ for some projection $q$ with $\tau(q)=1$ and 
$t\in \bR$. 

\smallskip

If $t>0$, this follows by homogeneity applying the result for $y/t$,
$x/t$ and $q$. If $t<0$, one just needs to exchange $x$ and $y$. The
constant $C_{p,\theta}$ is the same as in the previous cases.

\medskip

{\noindent \it Case 4:} We assume  $\M$ finite,  $x, y\in \M^{sa}$ and $y=x+\sum_{i=1}^n t_iq_i$ where $t_i\in \bR$ and $q_i$ are 
orthogonal projections with $\tau(q_i)=1$. 

\smallskip

Simply put $x_0=x$ and $x_k=x+\sum_{i=1}^k t_iq_i$ and use the $p/\theta$-triangular inequality and Case 3 for $x_{k-1}$ and $x_k$.
$$\| f(x)-f(y)\|_{p/\theta}^{p/\theta}\leq \sum_{i=1}^n\| f(x_{i-1})-f(x_{i})\|_{p/\theta}^{p/\theta}\leq C_{p,\theta}^{p/\theta} \sum_{i=1}^n 
\big(\|t_iq_i\|_p^\theta\big)^{p/\theta}.$$
But $\|x-y\|_p= \Big(\sum_{i=1}^n |t_i|^p\Big)^{1/p}$ and we get this case as $\|q_i\|_p=1$.

\medskip

{\noindent \it Case 5:} We assume  $\M$ finite, $x, y\in \M^{sa}$ and $y=x+\sum_{i=1}^n t_iq_i$ where $t_i\in \bR$ and $q_i$ are 
orthogonal projections with $\tau(q_i)\in \bQ$.

\smallskip

Consider the von Neumann algebra $(\tilde \M,\tilde \tau)$ where
$\tilde \M= \M\otimes L_\infty([0,1])$ with the trace $\tilde
\tau=N\tau\tens \int$ where $N$ is such that for all $i$, $N\tau(q_i)\in \bN$ .
Put $\tilde z=z\tens 1$ for $z\in \M$. We have $\tilde y-\tilde x=\sum_{i=1}^n 
t_i \tilde q_i$. For each $i$, and $k\leq N\tau(q_i)$, let $q_{i,k}=q_i\otimes
1_{[(k-1)/(N\tau(q_i)),k/(N\tau(q_i))]}$. The $(q_{i,k})_k$ are orthogonal projections 
with $\tilde \tau (q_{i,k})=1$ and $\sum_{k=1}^{N\tau(q_i)} q_{i,k}=\tilde q_i$.
By Case 4:
$$N^{\theta/p}\|f(x)-f(y)\|_{L_{p/\theta}(\M)}=\|f(\tilde x)-f(\tilde y)\|_{L_{p/\theta}(\tilde \M)}\leq C_{p,\theta}\|\tilde x-\tilde y\|_{L_{p}(\tilde \M)}^\theta= C_{p,\theta}\Big(
N^{1/p}\|x-y\|_{L_{p}(\M)}\Big)^\theta.$$
We still obtain the same constant $C_{p,\theta}$.

{\noindent \it Case 6:} We assume  $\M$  finite and $x, y\in \M^{sa}$.

\smallskip

By considering again $(\tilde \M,\tilde \tau)$ where $\tilde \M=
\M\otimes L_\infty([0,1])$ with the trace $\tilde \tau=\tau\tens
\int$, in a masa containing it, one can approximate $x-y$ for the
$L_1$-norm by elements of the form $\delta_k=\sum_{i=1}^{n_k}
t_{i,k}q_{i,k}$ where $\tilde \tau (q_{i,k})\in \bQ$. With 
$y_k=x+\delta_k$, $y_k$ converges to $y$ in $L_1$ it also
does in $L_{p}$ as $\tilde \M$ is finite. The Ando inequality for
the index $1/\theta$ also gives that $\|f(y_k)-f(y)\|_{1/\theta}\leq C_\theta\|y_k-y\|_1^\theta$, hence $f(y_k)$ converges to $f(y)$ in $L_{p/\theta}$. Since 
$x,y_k$ satisfy the assumptions of Case 5, the conclusion follows by
going to the limit in \eqref{main}.

\medskip

{\noindent \it Case 7}: We assume $\M$ semi-finite and $x,y\in L_p(\M)^{sa}$.

\smallskip

 This is again a matter of approximation.  By the semi-finiteness of
 $\M$ and functional calculus in a masa containing $x$, we may
 approximate $x$ in $L_p$ by elements of the form $x_n=\sum_{i=1}^N
 t_iq_n$ commuting with $x$ where $q_n$ are finite projections. We can
 as well assume that $f(x_n)$ converges to $f(x)$ in $L_{p/\theta}$,
 similarly for $y_n$ and $y$ (these are purely commutative
 results). Then $y_n$ and $x_n$ are in some finite subalgebra of $\M$
 and we have \eqref{main} for them by Case 6:
 $\|f(x_n)-f(y_n)\|_{p/\theta}\leq C_{p,\theta}
 \|x_n-y_n\|_p^\theta$. Taking again limits as $n\to \infty$ gives the
 result.
\end{proof}

\begin{rk}
{\rm One can get a proof of the case $p=\infty$ following case 1 (since $p/2=\infty$) and then case 6 directly. A careful analysis of the constant shows that $\limsup_{\theta\to 1} (1-\theta)C_{\infty,\theta}<\infty$ as in \cite{AP3}.}
\end{rk}

We slightly extend the result to
  $\tau$-measurable operators $L_0(\M,\tau)$. With its measure topology
  \cite{terp, FK}, it becomes a complete Hausdorff topological $*$-algebra (Theorem 28 in \cite{terp}) in which $\M$ is dense.
 We recall the Fatou Lemma in $L_0$ (Lemma 3.4 in
  \cite{FK}) if $v_n\to v$ in $L_0$ and then 
$\|v\|_q\leq \liminf \|v_n\|_q$ for all $q$. Moreover, the functional 
calculus associated
  to our $f(t)=|t|^\theta$ or $f(t)={\rm sgn} (t) |t|^\theta$ is continuous on $L_0^{sa}$ (Lemma 3.2 in \cite{Ray}).

\begin{thm}\label{ando2+}
Let $0<\theta<1$ and $0<p\leq \infty$ then there exists $C_{p,\theta}$ so that for any  semi-finite von Neumann algebra $(\M,\tau)$, and $x,y\in L_0(\M,\tau)^{sa}$ such that $x-y\in L_p(\M,\tau)$, then $|x|^\theta-|y|^\theta,{\rm sgn} (x) |x|^\theta-{\rm sgn} (y) |y|^\theta\in L_{p/\theta}(\M,\tau)$ and
\begin{equation*}\label{main2}\| |x|^\theta-|y|^\theta\|_{p/\theta}\leq C_{p,\theta}
\|x-y\|^\theta_p, \qquad \|{\rm sgn} (x) |x|^\theta-{\rm sgn} (y) |y|^\theta
\|_{p/\theta}\leq
C_{p,\theta} \|x-y\|^\theta_p.\end{equation*}
\end{thm}

\begin{proof}
This is a matter of approximations.

 We start by giving arguments to get the result when
 $x,y\in \M^{sa}$ and $x-y\in L_p(\M)$.  First note that if a
 bounded sequence $(x_n)_n\in \M^{sa}$ goes to $x$ for the strong-topology
then $(f(x_n))_n$ also goes to $f(x)$ for the strong-topology (Lemma 4.6 
in \cite{Tak}). Take any finite projection $\gamma\in \M$ and $(q_n)_n$ a sequence of finite projections that goes to 1 strongly, it follows that 
$(\gamma f(q_nxq_n)\gamma)_n$ goes to  $\gamma f(x)\gamma$ in $L_2$.
since $L_2\subset L_0$ is continuous, we get that in the topology of $L_0$, $\lim_n \gamma f(q_nxq_n)\gamma
=\gamma f(x)\gamma$ and similarly for $y$.

 As $q_nxq_n$ and $q_nyq_n$ are in $L_p$, we get 
$$\| \gamma (f(q_nxq_n)-f(q_nyq_n))\gamma\|_{p/\theta}\leq C_{p,\theta} \| q_n(x-y)q_n\|_p^\theta\leq C_{p,\theta} \|x-y\|_p^\theta.$$
Using the Fatou lemma, we obtain
$$\| \gamma (f(x)-f(y))\gamma\|_{p/\theta}\leq C_{p,\theta} \|x-y\|_p^\theta.$$
It is a simple exercise to check that if $z\in L_0^{sa}$ is so that 
$\sup \|\gamma z \gamma \|_{p/\theta} \leq C$ where the $\sup$ runs over all finite projections, then $z\in L_{p/\theta}$ with norm less than $C$. This allows to conclude.

Next take any $x,y\in L_0$ such that $x-y\in L_p$. Let $q_n=1_{|x|\leq
  n}\wedge 1_{|y|\leq n}$, then $q_n$ is an increasing sequence of projections
to 1 with $\tau(1-q_n)\to 0$ as $x,y\in L_0$. Moreover $q_nxq_n\in \M$ goes to $x$ in 
$L_0$ (similarly for $y$) (see \cite{terp} page 20). By the continuity of the functional calculus in 
$L_0$,  $(f(q_nxq_n))$ and $(f(q_nyq_n))$ go to $f(x)$ and $f(y)$.
Once again with the help of the Fatou lemma 
and the previous case  
$$\|f(x)-f(y)\|_{p/\theta}\leq \liminf \|f(q_nxq_n)-f(q_nyq_n)\|_{p/\theta}
\leq C_{p,\theta} \liminf \|q_n(x-y)q_n\|_p^\theta \leq C_{p,\theta}\|x-y\|_p^\theta.$$
\end{proof}

\section{Ando's inequality in type III algebras}

Before going to type $III$ algebras, we need to extend Ando's
inequality to weak-$L_p$ spaces. We will use the $K$-interpolation
method see \cite{BL}.

\medskip

As usual, given two compatible quasi-Banach
 spaces $X_0$ and $X_1$, we let for $t>0$ and $x\in X_0+X_1$
$$K_t(x,X_0,X_1)=\inf\{ \|x_0\|_{X_0}+ t\|x_1\|_{X_1}\;;\; x=x_0+x_1\}.$$

For $0<\eta<1$ and $0<q\leq\infty$ set 
$$\|x\|_{\eta,q}= \| t^{-\eta} K_t(x,X_0,X_1)\|_{L_q(\bR^+,dt/t)}=\left(\int_0^\infty
(t^{-\eta} K_t(x,X_0,X_1))^q \frac {dt}t\right)^{1/q},$$
with the obvious modification when $q=\infty$.

The interpolated space $(X_0,X_1)_{\eta,q}$ is $\{x\in X_0+X_1 \; |\; \|x\|_{\eta,q}<\infty\}$ with (quasi)-norm $\|.\|_{\eta,q}$.

 If $(\M,\tau)$ is a semi-finite von Neumann algebra, and $x\in L_0(\M,\tau)$ we denote as usual its decreasing rearrangement by $\mu_t(x)$ (see \cite{FK}).

 The noncommutative Lorentz spaces $L_{p,q}(\M,\tau)$ for $0<p<\infty$
 and $0<q\leq \infty$ are defined as in the commutative case.
 The space $L_{p,q}(\M,\tau)$ consists of all measurable 
operators $x\in L_0(\M,\tau)$ so that $\|x\|_{L_{p,q}}=\|t^{1/p}
 \mu_t(f)\|_{L_q(\bR^+,dt/t)}<\infty$. With the (quasi)-norm
 $\|.\|_{L_{p,q}}$, it becomes a (quasi)-Banach space.

The results about real interpolation of commutative $L_p$-spaces
(\cite{BL} Theorem 5.3.1) remain available for semi-finite von Neumann
algebras. Indeed for any $0<p_0<p_1\leq \infty$, for all $t>0$, $x\in
L_{p_0}(\M,\tau)+L_{p_1}(\M,\tau)$, the quantities
$K_t(x,L_{p_0}(\M,\tau),L_{p_1}(\M,\tau))$ and
$K_t(\mu(x),L_{p_0}(\bR^+), L_{p_1}(\bR^+))$ are
        equivalent with constants depending only on $p_0$ and $p_1$,
        see \cite{Xu} for details. As a consequence, we get that for
        $p_0\neq p_1\leq \infty$ and $0<q\leq \infty$, $0<\eta<1$ with
        $\frac 1 p=\frac {1-\eta}{p_0}+\frac \eta {p_1}$,
        $(L_{p_0}(\M,\tau),L_{p_1}(\M,\tau))_{\eta,q}=L_{p,q}(\M,\tau)$
        with equivalent norms (depending only on the parameters but
        not on $(\M,\tau)$).

We drop the reference to $(\M,\tau)$ to lighten notation.

 When $x=x^*\in L_{p_0}+L_{p_1}$, we can also consider
$$K^{sa}_t(x,L_{p_0},L_{p_1})=\inf \{ \|x_0\|_{p_0}+ t\|x_1\|_{p_1}\;;\; x=x_0+x_1 \textrm{ with } x_i=x_i^*\}.$$

 \begin{lemma}
Let $0<p_0<p_1\leq \infty$, $t>0$ and $x\in (L_{p_0}+L_{p_1})^{sa}$, then 
$$K_t^{sa}(x,L_{p_0},L_{p_1})\geq K_t(x,L_{p_0},L_{p_1})\geq K_t^{sa}(x,L_{p_0},L_{p_1})/2^{\max\{1/p_0,1\}-1}.$$
\end{lemma}
\begin{proof}
This is a standard fact. Let $x=a_0+a_1$ with $a_i\in L_{p_i}$ such that 
$\|a_0\|_{p_0}+t\|a_1\|_{p_1}\leq K_t(x,L_{p_0},L_{p_1})+\epsilon$. Then
$x=b_0+b_1$ with $b_i=(a_i+a_i^*)/2$. We have $\|b_i\|_{p_i}\le 2^{1/p_i-1}\|a_i\|_{p_i}$ if $p_i<1$ or $\|b_i\|_{p_i}\le \|a_i\|_{p_i}$ otherwise, hence we get
the result letting $\eps\to 0$.
\end{proof}

\begin{rk}{\rm
Similarly one can easily show that for $x\geq 0$, $K_t(x,L_{p_0},L_{p_1})$
is equivalent to $\inf \{ \|x_0\|_{p_0}+ t\|x_1\|_{p_1}\;;\; x=x_0+x_1 \textrm{ with } x_i\geq 0\}.$}
\end{rk}

\begin{lemma}\label{kfonc}Let $0<p_0<p_1\leq \infty$ and $\theta\in]0,1[$ and $x,y \in (L_{p_0}+L_{p_1})^{sa}$. 
Then for all $t>0$ and $f(s)=|s|^\theta$ or $f(s)={\rm sgn}(s)|s|^{\theta}$:
$$K_{t^\theta}(f(y)-f(x),L_{p_0/\theta},L_{p_1/\theta})\leq C_{p_0,p_1,\theta} K_t(y-x,L_{p_0},L_{p_1})^\theta.$$
\end{lemma}
\begin{proof}
Thanks to the previous Lemma, it suffices to do it with $K_t^{sa}$ instead of $K_t$.

Choose selfadjoint operators
 $\delta_0\in L_{p_0}$ and $\delta_1\in L_{p_1}$ so that 
$y-x=\delta_0+\delta_1$ and  $\|\delta_0\|_{p_0}+t\|\delta_1\|_{p_1}\leq 
2K^{sa}_{t}(x-y,L_{p_0},L_{p_1})$. 

Set $a_0=f(y)-f(x+\delta_1)=f(y)-f(y-\delta_0)$ and
$a_1=f(x+\delta_1)-f(x)$, then $f(y)-f(x)=a_0+a_1$. By Theorem
\ref{ando2+} for $p_i$, we obtain $\|a_i\|_{p_i/\theta}\leq
C_{p_i,\theta} \|\delta_i\|_{p_i}^\theta$. Since, $\|a_0\|_{p_0/\theta} + t^\theta
\|a_1\|_{p_1/\theta}\leq
C_{p_0,p_1,\theta}(\|\delta_0\|_1+t\|\delta_1\|_{p_1})^\theta$,
we have found a suitable
decomposition to conclude.
\end{proof}

\begin{prop}\label{weak}
For all $0<p<\infty$, $0<q\leq \infty$ and $0<\theta<1$, there exists $C_{p,q,\theta}>0$ such that for $x,y\in  L_{p}^{sa}$, with $f(s)=|s|^\theta$ or $f(s)={\rm sgn}(s)|s|^{\theta}$:
$$\|f(y)-f(x)\|_{L_{p/\theta,q}}\leq C_{p,q,\theta} \|y-x\|_{L_{p,q\theta}}^\theta.$$
\end{prop}
\begin{proof}
Put $p_0=p/2$, $p_1=2p$ and $\eta=2/3$ so that $\frac 1 p= \frac
{1-\eta}{p_0}+\frac \eta {p_1}$. We have
\begin{eqnarray*}
\|f(y)-f(x)\|_{L_{p/\theta,q}} &\simeq_{p/\theta,q}& \| t^{-\eta} K_t(f(y)-f(x),L_{p_0/\theta},L_{p_1/\theta})\|_{L_q(\bR^+,dt/t)},\\
 \|y-x\|_{L_{p,q\theta}} &\simeq_{p,q\theta}&\| u^{-\eta} K_{u}(y-x,L_{p_0},L_{p_1})\|_{L_{q\theta}(\bR^+,du/u)}.\end{eqnarray*}
But by Lemma \ref{kfonc}
$$\| t^{-\eta}
K_t(f(y)-f(x),L_{p_0/\theta},L_{p_1/\theta})\|_{L_q(\bR^+,dt/t)}\leq
C_{p,\theta}\| t^{-\eta}
K_{t^{1/\theta}}(y-x,L_{p_0},L_{p_1})^\theta\|_{L_q(\bR^+,dt/t)}.$$ But
by a change of variable $u=t^{1/\theta}$: $$\| t^{-\eta}
K_{t^{1/\theta}}(y-x,L_{p_0},L_{p_1})^\theta\|_{L_q(\bR^+,dt/t)}=\theta^{1/q}\|
u^{-\eta} K_{u}(y-x,L_{p_0},L_{p_1})\|^\theta_{L_{q\theta}(\bR^+,du/u)}$$
and we get the estimate.\end{proof}

We can conclude with the proof of Theorem \ref{ando2} for type III algebras.
\begin{proof}
 Assume that $\M$ is given with a
n.s.f. weight $\phi$ with modular group $\sigma$. Let
$\R=\M\rtimes_{\hat\sigma} \bR$ be its core (with the dual action
$\hat\sigma$), this is a semi-finite von Neumann algebra with a trace $\tau$
such that $\tau\circ \hat\sigma_t=e^{-t}\tau$. Then by definition $L_{p}(\M,\phi)$ is
isometrically a  subspace of $L_{p,\infty}(\R, \tau)$, more precisely
$$L_p(\M,\phi)=\{ x\in L_0(\R,\tau)\; |\; \hat\sigma_t(x)=e^{-t/p}x,
\,\forall t\in \bR\}\subset L_{p,\infty}(\R, \tau),$$ 
with by definition $\|x\|_p=\|x\|_{L_{p,\infty}(\R)}$ (see \cite{terp} Definition 13 and Lemma 5 or Lemma B in \cite{Kos2}).

Thus Theorem
\ref{ando2} for type III algebras follows from Proposition \ref{weak} for the
weak-$L_p$ spaces noticing that $f(L_p(\M,\phi))\subset L_{p/\theta}(\M,\phi)$ as $\hat\sigma_t$ is a representation and $f$ is $\theta$-homogeneous.
\end{proof}

\section{Further comments}

By very classical arguments using the Cayley transform see \cite{R},
estimates for the functional calculus are equivalent to some for
commutators or anticommutators.

\begin{prop}
Let $0<p\leq\infty$ and $0<\theta<1$, then there exists a constant $C_{p,\theta}$ so that for $x \in L_{p}(\M)^{sa}$ and $b\in \M$ with $f(s)=|s|^\theta$ or $f(s)={\rm sgn}(s)|s|^{\theta}$:
$$\big\|[f(x),b]\big\|_{p/\theta} \leq C_{p,\theta} \big\|[x,b]\big\|_p^\theta\|b\|^{1-\theta}.$$
If $x,y \in L_{p}(\M)^{+}$ and $b\in \M$, then 
$$\big\| bx^\theta\pm y^\theta b\big\|_{p/\theta}\leq C_{p,\theta} \big\|bx\pm yb\big\|_p^\theta \|b\|^{1-\theta}.$$
\end{prop}

Let $M_{p,q}$ denote the Mazur map for $p<q$ given by
$M_{p,q}(f)=f|f|^{(p-q)/q}$. The $2\times 2$-tricks from \cite{R} also give

\begin{prop}
Let $0<p\leq \infty$ and $0<\theta<1$, then there exists a constant $C_{p,\theta}$ so that for $x,y \in L_{p}(\M)$:
$$\big\|M_{p,p/\theta}(x)-M_{p,p/\theta}(y)]\big\|_{p/\theta} \leq C_{p,\theta} \|x-y\|_p^\theta.$$
\end{prop}
This improves the estimates for $M_{p,q}$ when $p<q$ of Theorem 4.1 in \cite{PR} and gives the expected H\"older continuity.

\begin{rk}
{\rm The above two propositions are valid for all von Neumann algebras. For semi-finite ones, similar estimates are true for the Lorentz norms as in Proposition \ref{weak}.}
\end{rk}

\bigskip

Contrary to the case $p\geq 1$, at least for $p\leq 1/2$,
  there is no constant $C_p$ so that for all $0<\theta<1$, 
$\|\Big(\frac {x^\theta- y^\theta}{x-y}\Big)_{x\geq 1,y\geq 1}\|_{pcb}\leq C_p.$

Indeed if this is so, then taking $\theta=1-\eps$, $x=e^{t/\eps}$,
$y=e^{s/\eps}$ by Lemma \ref{approx}, letting $\eps\to 0$, we would
get that $\|\Big(e^{-\max\{t,s\}}\Big)_{t,s\geq 0}\|_{pcb}\leq
C_p$. Using that $2\max\{x,y\}= |x-y|+x+y$, we would deduce that
$\|\Big(e^{-|t-s|}\Big)_{0\leq t,s\leq 1}\|_{pcb}\leq
C_p$. 

By easy arguments going from a discrete situation to a continuous
one, one can deduce that the same matrix has to be a Schur multiplier on 
$S^p(L_2[0,1])$. The constant kernel 1 on $[0,1]^2$  is in 
 $S^p(L_2[0,1])$ with norm one, thus we would get that 
$\|\Big(e^{-|x-y|}\Big)_{0\leq x,y\leq 1}\|_{S^p(L_2[0,1])}\leq C_p.$

The operator $T$ on $L_2[0,1]$ with kernel $K(x,y)=e^{-|x-y|}$ is
obviously positive and Hilbert-Schmidt with HS norm less than 1.  One
easily checks that the eigenvectors $f_\lambda$ associated to
$\lambda$ must satisfy $\lambda f''=\lambda f-2f$. Letting
$\lambda=\frac 2{1+\alpha^2}$ with $\alpha> 0$, the only possibilities
are $f_\lambda(x)=ae^{{\rm i} \alpha x}+be^{-{\rm i} \alpha x}$. But
$T(e^{{\rm i} \alpha x})= \frac 2 {1+\alpha^2} e^{{\rm i} \alpha x} -
\frac{e^{-x}}{1+ {\rm i} \alpha}-e^x \frac{e^{-1+{\rm i} \alpha}}{1-
  {\rm i} \alpha}$.  Thus, $\lambda$ is a eigenvalue iff $e^{2{\rm i}
  \alpha}=\Big(\frac{ 1-{\rm i} \alpha}{1+{\rm i} \alpha}\Big)^2$. Set
$\tan \theta=\alpha$ with $\theta\in]0,\pi/2[$, so that $e^{2{\rm i}
    \alpha}=e^{-4{\rm i \theta}}$. The equation $\tan
  t=-2t+k\pi$ admits a unique solution $\theta_k$ in $]0,\pi/2[$ for
  $k>0$ so that $\tan \theta_k\approx k\pi$. We deduce that the set of eigenvalues of $T$ is $(\frac
  2{1+\tan (\theta_k)^2})_{k\ge 1}$ with associated eigenvector
  $f_\lambda(x)= e^{{\rm i} \alpha x}-\frac{1+{\rm i} \alpha}{1-{\rm i} \alpha}
e^{-{\rm i} \alpha x}$. Thus $T\notin S^p (L_2[0,1])$ when $p\leq 1/2$.

As  $f_\lambda/\|f_\lambda\|_2$ is uniformly bounded in
$\C[0,1]$, we can also deduce using Lemma \ref{schur1} that  for $p>1/2$
$$\|\Big(e^{-\max\{t,s\}}\Big)_{0\leq t,s\leq 1}\|_{pcb}<\infty \quad\textrm{and}\quad\|\Big(e^{-\max\{t,s\}}\Big)_{t,s\geq 0}\|_{pcb}<\infty.$$

About Corollary 2.11, it is easier to see that the norm cannot be
independent of $0<\theta<1$ for $p\leq 1$. Indeed the multiplier
$\Big(\frac {x-y}{x+y}\Big)_{x\geq 1,y\geq 1}$, or equivalently
$\Big(\frac {x-y}{x+y}\Big)_{x\geq 0,y\geq 0}$ using homogeneity, is
not bounded for $p=1$ (hence also for $p<1$) as shown in \cite{Dav}.

\bibliographystyle{plain}

\end{document}